\documentclass[a4paper,10pt,reqno, english]{amsart}

\usepackage{amsmath,amssymb,amscd,amsthm,amsfonts}
\usepackage{graphicx,subfigure}
\usepackage{hyperref}
\usepackage{dsfont}
\usepackage[nobysame, alphabetic]{amsrefs}
\usepackage{tikz}
\usepackage[capitalise]{cleveref}
\usepackage{mathrsfs}
\usepackage{enumitem}

\newtheorem{theorem}{Theorem}
\newtheorem{lemma}{Lemma}

\newtheorem{corollary}{Corollary}
\newtheorem{example}{Example}

\newtheorem{definition}{Definition}
\newtheorem{remark}{Remark}

\newcommand{\R}{\mathds{R}}
\newcommand{\Z}{\mathds{Z}}

\newcommand{\Hc}{\mathcal H}
\newcommand{\Rc}{\mathcal R}

\def\rr{\mathds{R}}

\DeclareMathOperator{\conv}{conv}
\DeclareMathOperator{\hs}{hs}
\DeclareMathOperator{\pt}{pt}

\newcommand{\disc}{\operatorname{disc}}
\newcommand{\herdisc}{\operatorname{herdisc}}
\newcommand{\lindisc}{\operatorname{lindisc}}
\newcommand{\Sym}{\operatorname{Sym}}

\newcommand{\im}{\operatorname{im}}
\newcommand{\tr}{\operatorname{tr}}

\newcommand{\spr}{\operatorname{spr}}

\newcommand{\hooklongrightarrow}{\lhook\joinrel\longrightarrow}

\title{Discrepancy theory, Tverberg's theorem, and regression depth}

\hypersetup{
  pdftitle={Discrepancy theory, Tverberg's theorem, and regression depth},
  pdfauthor={}
}

\author[Lopez]{Aleksey Lopez}\address{University of Michigan, Ann Arbor, MI 48109, United States}
\email{lopezag@umich.edu}

\author[Sober\'on]{Pablo Sober\'on}\address{Baruch College, City University of New York, One Bernard Baruch Way, New York, NY 10010, United States} 
\email{psoberon@gc.cuny.edu}

\thanks{The research of A. Lopez was supported by NSF grant DMS-2349366 and by Jane Street.  The research of P. Sober\'on is supported by NSF CAREER grant DMS-2237324 and a PSC-CUNY Track 1 award.}



\begin{document}

\begin{abstract}
We prove new bounds for Tverberg's theorem with tolerance.  We show that $N = rt+\Theta_{d,r}(t^{1/2-1/(2d)})$, where $N$ is the smallest number such that any set of $N$ points in $\mathbb{R}^d$ has a partition into $r$ parts such that the convex hulls of the parts intersect even if we remove any $t$ of the points.

We extend Tverberg's theorem with tolerance to families of hyperplanes in $\rr^d$, and show that for any set of $rt + O_{d,r}(t^{1/2-1/(2d)}\sqrt{\log (t+1)})$ hyperplanes in $\rr^d$ there exists a partition of them into $r$ parts such that the regression hulls of the parts intersect even if any $t$ hyperplanes are removed.

Our bounds follow from establishing a connection between Tverberg-type results and discrepancy theory.
\end{abstract}

\maketitle

\section{Introduction}

Given sufficiently many points in Euclidean space, Tverberg’s theorem guarantees that they can be divided into a prescribed number of groups whose convex hulls all contain a common point.  More precisely, it states that \textit{given $(r-1)(d+1)+1$ points in $\rr^d$, they can be split into $r$ parts whose convex hulls intersect} \cite{Tverberg1966}.  Such a partition is called a Tverberg partition.  Tverberg's theorem is a key result on the combinatorial properties of finite families of points in Euclidean spaces, with many variations and extensions \cites{Barany2018, DeLoera2019, Blagojevic2017}.

One particular variation of Tverberg's theorem, known as Tverberg with tolerance, asks how many points we need in $\rr^d$ to have a Tverberg partition that resists the removal of any $t$ points.  Formally, given a set $X$ of points in $\rr^d$, we want to split them into sets $X_1,\dots, X_r$ such that $\bigcap_{j=1}^r \conv (X_j \setminus C)\neq \emptyset$ for any $C \subset X$ of at most $t$ points. The parameter $t$ is the tolerance.  This question was introduced by Larman, when he proved that $2d+3$ points in $\rr^d$ always have a partition into two sets with tolerance $1$ \cite{Larman1972}.  We denote by $N_{\pt}(t,d,r)$ the smallest number of points that guarantee the existence of such a partition.

The exact value of $N_{\pt}(t,d,r)$ is only known when $t=0$ (given by Tverberg), when $d=1$ \cite{Mulzer2014}, a handful of other particular cases when $d=2$ \cite{Bereg2020}, $N_{\pt}(1,3,2)=9$ \cite{Bereg2020}, and $N_{\pt}(1,4,2)=11$ \cite{Forge2001}.  Several variations of this problem have been studied, including colorful and algorithmic versions \cites{Montejano2011, Soberon2012, Soberon2015, Sarkar2022, Bereg2022}.  One of the biggest surprises concerning this problem is likely the bounds found by Garc\'ia-Col\'in, Raggi, and Rold\'an-Pensado.  They showed that if $r,d$ are fixed, then $N_{\pt}(t, d, r) = rt + o(t)$ \cite{GarciaColin2017}. In other words, as $t$ becomes large the effect of the dimension mostly vanishes.  This was improved by the second author to $N_{\pt}(t, d, r) = rt + \tilde{O}_{d,r}(\sqrt{t})$, where the $\tilde{O}_{d,r}(\cdot)$ notation hides terms that are polynomial in $d,r$ and polylogarithmic in $t$ \cite{Soberon2018}.  Our first result is a proof of new upper and lower bounds on $N_{\pt}(t, d, r)$.  In particular, we determine the dependence on $t$ of $N_{\pt}(t, d, r) - rt$.  We state the theorem below for $d\ge 2$ since $N_{\pt}(t,1,r) = r(t+2)-1$ \cite{Mulzer2014}.

\begin{theorem}\label{thm:main-points}
    For positive integers $t, d, r$ with $d, r \ge 2$, we have
    \[
    N_{\pt}(t, d, r) = rt + \Theta_{d,r}\left(t^{1/2 - 1/(2d)}\right).
    \]
\end{theorem}

Our second main result concerns Tverberg with tolerance for families of hyperplanes.  To establish analogues of Tverberg's theorem for families of hyperplanes, we need an extension of the convex hull.  Given a set $X$ of points in $\rr^d$, a point $p$ is in $\conv(X)$ if and only if every closed halfspace that contains $p$ also contains at least one point of $X$.  For a set $X'$ of hyperplanes in $\rr^d$, we will denote by $D(X')$ its \textit{regression hull}.  A point $p$ is in $D(X')$ if and only if every closed ray starting from $p$ intersects or is parallel to at least one hyperplane in $X'$.  We choose this name because it is closely related to linear regression and regression depth \cites{Rousseeuw1999, Rousseeuw1999a, Amenta2000, Kreveld2008, Fulek2009}.

There are several Tverberg-type results for regression depth \cite{Karasev2011, Karasev2014, Schnider2023a}.  We are not aware of a Tverberg with tolerance theorem for regression depth.  We prove an upper bound which, up to a logarithmic factor in the tolerance, behaves the same way as our new bounds for Tverberg's theorem with tolerance.  Formally, let $N=N_\text{hyp}(t, d, r)$ be the smallest number of hyperplanes such that for every family $X$ of $N$ hyperplanes in $\rr^d$ there exists a partition $X_1,\dots,X_r$ of them into $r$ parts such that $\bigcap_{j=1}^r D(X_j \setminus C) \neq \emptyset$ for all subsets $C \subset X$ of at most $t$ hyperplanes.  Notice that $N_{\text{hyp}}(t,1,r) = N_{\pt}(t,1,r)$, so the open cases start with $d \ge 2$.

\begin{theorem}\label{thm:main-hyperplanes}
    Let $t,r,d$ be positive integers with $d\ge 2,\ r \ge 2$.  Then, 
\[
N_\text{hyp}(t, d, r) = rt + O_{d,r}\left(t^{1/2-1/(2d)}\sqrt{\log (t+1)}\right).
\]
\end{theorem}

\begin{remark}
    The constants in the upper bounds can be taken to be polynomial in $d, r$. Similarly, the constant in the lower bound can be taken to be the reciprocal of a polynomial in $d, r$. Establishing this dependence requires careful tracking of constants in the discrepancy proofs of the results from \cites{Matousek1995, Matousek1993} used in our arguments.  We omit these details, as they amount to straightforward bookkeeping in previous papers.  It would add substantial length without introducing new ideas.
\end{remark}

The tools we use in our proofs are based on discrepancy theory.  Discrepancy measures the unavoidable error in approximating some distributions by finite samples or in satisfying many competing balancing constraints. It has proved useful in a wide range of topics, including combinatorics, discrete and computational geometry, number theory, and probability \cites{Beck1987, Matousek2010, Chazelle2000}.  The earlier advances that show $N_{\pt}(t,d,r) = rt + \tilde{O}_{d,r}(\sqrt{t})$ use probabilistic arguments.  The discrepancy theory bounds from Matou\v{s}ek \cite{Matousek1995} and from Matou\v{s}ek, Welzl, and Wernisch \cite{Matousek1993} allow us to improve the exponent in our results to $1/2-1/(2d)$.  See \cite{adibelli2026discrepancy} for recent progress in geometric discrepancy theory.

\section{Notation and results from discrepancy theory}

\subsection{Classic discrepancy theory}

We first introduce some notions and results from discrepancy theory that will be useful in our proofs.

A finite range space is a pair $(X,\Rc)$, where $X$ is a finite set and $\Rc \subseteq 2^X$ is a family of subsets of $X$.  We call the elements of $\Rc$ ranges.  For a subset $Y \subseteq X$, the induced range on $Y$ is $\Rc_Y = \{R \cap Y: R \in \Rc\}$.

A significant part of discrepancy theory deals with range spaces of bounded complexity.  A standard way to measure the complexity of a range space is by its shatter and dual shatter functions.

\begin{example}
    Let $X \subseteq \rr^d$ be a finite set.  The closed halfspaces of $\rr^d$ induce a range space on $X$ by defining \(\Rc_{\hs} = \{X \cap H: H \mbox{ is a closed halfspace of }\rr^d\}\).
\end{example}

Given a finite range space $(X,\Rc)$, its shatter function $\pi_{\Rc}: [|X|]\to \mathds{N}$ where 
\[
\pi_{\Rc}(m) = \max_{Y \in \binom{X}{m}} |\{Y \cap R : R \in \Rc \}|.
\]
In other words, it is the largest number of subsets of any $m$-element subset of $X$ that we can obtain by intersecting it with the ranges.  The VC dimension of $(X, \Rc)$ is the maximum $m$ such that $\pi_{\Rc}(m) = 2^m$.  In other words, the largest size of a set that we can shatter completely using $\Rc$.

For example, for $\Rc_{\hs}$ in $\rr^d$, we know that the maximum number of regions that $m$ hyperplanes split $\rr^d$ is bounded above by $2\sum_{j=0}^d \binom{m-1}{j}$ which implies $\pi_{\Rc_{\hs}}(m) \le (4m)^d$ \cite{Matousek2002}.  The VC dimension of $\Rc_{\hs}$ is at most $d+1$.

To define the dual shatter function, it is convenient to consider equivalence relations with respect to ranges.  Given a set of ranges, we say that $x, y \in X$ are equivalent if they are contained in the same set of ranges.  Then, the dual shatter function $\pi^*_{\Rc}: [|X|] \to \mathds{N}$ is the function such that $\pi^*_{\Rc}$ is the maximum possible number of equivalent classes in $X$ we obtain by choosing $m$ ranges $R_1,\dots,R_m$.  The dual VC dimension of the range $(X,\Rc)$ is the largest value $m$ for which $\pi^*_{\Rc}(m)=2^m$.  In other words, it is the largest possible number of ranges we can find such that, for any subset of those ranges, there is some $x \in X$ that is contained in those ranges and not contained in the rest.

Discrepancy measures the forced imbalances in range spaces.  In other words, the discrepancy of a range space should tell us how much imbalance is unavoidable whenever we color $X$ with two colors.  There are many notions of discrepancy.  We include below the most common ones.

\begin{definition}\label{def:two-disc}
    Suppose $(X, \Rc)$ is a finite range space, and let $\chi: X \to \{-1, +1\}$ be a two-coloring of $X$. The signed imbalance of a range $R \in \Rc$ is
    $$\chi(R) := \sum_{x \in R} \chi(x).$$

    \begin{itemize}
        \item The two-color discrepancy of $\Rc$ is
    $\displaystyle \disc(\Rc) := \min_{\chi} \max_{R\in\Rc} \lvert \chi(R) \rvert$.
        \item The hereditary discrepancy of $\Rc$ is $\herdisc(\Rc) := \max_{Y \subseteq X} \disc(\Rc_Y)$.
        \item The linear discrepancy of $\Rc$ is
    $\displaystyle \lindisc(\Rc) := \max_{w \in [0, 1]^X} \min_{z \in \{0, 1\}^X} \max_{R \in \Rc} \left\lvert \sum_{x \in R} (z_x - w_x) \right\rvert$.
    \end{itemize}
\end{definition}

An interpretation of linear discrepancy is that an adversary chooses weights $w_x \in [0,1]$ for all $x \in X$.  We wish to approximate all values $\sum_{x \in R} w_x$ by choosing instead boolean weights $z_x$.  The linear discrepancy of the system gives a lower bound for the maximum absolute value of the difference between the value we want to approximate and the value we obtain.  When $w_x = 1/2$ for all $x \in X$, we recover $1/2$ of the two-color discrepancy of the system.  In general, linear discrepancy is bounded as follows.

\begin{theorem}[Lov\'asz, Spencer, Vesztergombi 1986 \cite{Lovasz1986}]\label{thm:lsv}
    For any finite range space $(X, \Rc)$,
    \[
    %
    \frac{1}{2}\disc(\Rc) \leq \lindisc(\Rc) \leq \herdisc(\Rc).
    \]
\end{theorem}

The discrepancy of a system can also be bounded in terms of its shatter function. In particular, Matou\v{s}ek proved the following bound \cite{Matousek1995}.

\begin{theorem}\label{thm:mat-expl}
    Let $A > 0$ and $s > 1$ be constants and let $(X, \Rc)$ be a finite range space with $n = |X|$. Suppose that $\pi_\Rc(m) \leq Am^s$ for all $1 \leq m \leq n$. Then
    $$\disc(\Rc) = O\left(n^{1/2 - 1/(2s)}\right).$$
\end{theorem}

The inequality involving the primal shatter function is preserved under restriction, so the same bound holds for hereditary discrepancy. In particular, we have the following corollary

\begin{corollary}\label{cor:herdisc-halfspaces}
Let $d \ge 2$ be an integer and $X \subset \rr^d$ be a finite set, then,
\[
\herdisc (\Rc_{\hs}(X))\le C_d |X|^{1/2-1/(2d)},
\]
where $C_d$ is a constant depending only on $d$.
\end{corollary}

The bound $\herdisc (\Rc_{\hs}(X)) = O(|X|^{1/2-1/(2d)})$ was proved by Matou\v{s}ek \cite{Matousek1995}.

Similarly to \cref{thm:mat-expl}, if we know the order of growth of the dual shatter function of a range space, we can bound the order of growth of its discrepancy, as shown by Matou\v{s}ek, Welzl, and Wernisch \cite{Matousek1993}.

\begin{theorem}\label{thm:mww-expl}
    Let $A > 0$ and $s > 1$ be constants and let $(X, \Rc)$ be a finite range space with $n = |X|$. Suppose that $\pi_\Rc^*(m) \leq Am^s$ for all $1 \leq m \leq n$. Then
    $$\disc(\Rc) = O\left(n^{1/2 - 1/(2s)}\sqrt{\log n}\right).$$
\end{theorem}

As above, the inequality involving the dual shatter function is preserved under restriction, so the same bound holds for hereditary discrepancy.

\subsection{Colorful discrepancy}

For our purposes, given a partition of the ground set into $r$ sets, we need to be able to measure how unbalanced a range is, even if $r>2$.  This is measured by multicolor discrepancy \cite{Doerr2003, Doerr2010}.

\begin{definition}
    If $X = X_1 \sqcup \cdots \sqcup X_r$ is an $r$-partition of the ground set of a finite range space $(X, \Rc)$, let $\chi: X \to [r]$ be the corresponding coloring, so that $X_i = \chi^{-1}(i)$. Its $r$-color discrepancy is
    \[
    \disc(\Rc, \chi, r) := \max_{\substack{R \in \Rc \\ i \in [r]}} \left\lvert \lvert R \cap \chi^{-1}(i) \rvert - \frac{\lvert R \rvert}{r} \right\rvert.
    \]
\end{definition}

We will also use a very similar parameter, which we call the $r$-color spread.

\begin{definition}
    Let $r$ be a positive integer and $(X,\Rc)$ be a range space.  If $X = X_1 \sqcup \cdots \sqcup X_r$ is a partition of $X$ with corresponding coloring $\chi: X \to [r]$, we define its $r$-color spread with respect to $\Rc$ to be
    \[
    \spr(\Rc, \chi, r) := \max_{\substack{R \in \Rc \\ i, j \in [r]}} \lvert \lvert R \cap \chi^{-1}(i) \rvert - \lvert R \cap \chi^{-1}(j) \rvert \rvert.
    \]
\end{definition}

These two notions are equivalent up to a factor of $2$.  Namely, if $\chi$ is the coloring corresponding to the partition $X = X_1 \sqcup \cdots \sqcup X_r$, then $\disc(\Rc, \chi, r) \le \spr(\Rc, \chi, r) \le 2 \cdot \disc(\Rc, \chi, r)$.   The most important bound from colorful discrepancy is that we can bound the $r$-color discrepancy in a range space in terms of its hereditary discrepancy.

\begin{theorem}[Doerr, Fouz 2010 \cite{Doerr2010}*{Thm 2.5, 2.6}]\label{lem:two-r-color}
    Let $r$ be a positive integer and let $(X, \Rc)$ be a finite range space. There is a partition $X = X_1 \sqcup \cdots \sqcup X_r$ with corresponding coloring $\chi: X \to [r]$ satisfying
    $$\disc(\Rc, \chi, r) \leq 2.0005\cdot \herdisc(\Rc).$$
\end{theorem}

The careful reader might notice that there is a missing factor $2$ in the Doerr--Fouz statement.  This is due to our normalization.

\section{Tverberg with tolerance}

\begin{lemma}\label{lem:dev-tolerance}
    Let $X \subseteq \rr^d$  be a finite set of points and let $X = X_1 \sqcup \cdots \sqcup X_r$ be a partition of $X$ with corresponding coloring $\chi: X \to [r]$. Let
    \(\Delta = \disc(\Rc_{\hs} (X), \chi, r)\).  Then, if $|X| > r^2 \Delta$, we have that $(X_1,\dots, X_r)$ is a Tverberg partition with tolerance greater than or equal to \( \left\lceil \frac{|X|}r - r\Delta \right\rceil - 1\).
\end{lemma}

\begin{proof}
    Let $C \subseteq X$ be a set such that $\bigcap_{j=1}^r \conv(X_j \setminus C) = \emptyset$.  We want to show that $|C| \ge \frac{|X|}{r}-r\Delta$.  Since the convex sets have empty intersection, there exist open halfspaces $H_1,\dots,H_r$ such that for all $j \in [r]$, we have $X_j \setminus C \subseteq H_j$ and $\bigcap_{j=1}^r H_j = \emptyset$.

    For each $j \in [r]$, let $G_j = \rr^d \setminus H_j$.  Note that since each point in $\rr^d$ is missing from at least one $H_j$, we have $\sum_{j=1}^r|X\cap G_j| \ge |X|$.

    By the discrepancy hypothesis, we have that for each $j$, $|X_j \cap G_j| \ge \frac{|X\cap G_j|}{r}-\Delta$.  Adding these inequalities together gives
    \[
    \sum_{j=1}^r |X_j \cap G_j| \ge \frac{1}{r}\sum_{j=1}^r|X\cap G_j| - r\Delta \ge \frac{|X|}{r}- r\Delta.
    \]
    By construction, $X_j \cap G_j \subseteq C$.  Moreover, since $X_1,\dots,X_r$ is a partition, every point of $C$ is counted at most once.  This implies $|C| \ge \frac{|X|}{r}-r\Delta$, as we wanted to show.
\end{proof}

With this, we can prove the upper bound for \cref{thm:main-points}.

\begin{theorem}\label{thm:point-bound}
    For positive integers $t, d, r$ such that $d \ge 2$, we have
    \[
    N_{\pt}(t, d, r) \leq rt + C_{d, r}t^{1/2 - 1/(2d)},
    \]
    where $C_{d, r}$ is a constant depending only on $d$ and $r$.
\end{theorem}

\begin{proof}
    Let $X$ be a set of $N$ points in $\rr^d$.  By \cref{lem:two-r-color} and \cref{cor:herdisc-halfspaces}, there exists a partition $X_1 \sqcup \dots \sqcup X_r$, with corresponding coloring $\chi: X \to [r]$, such that 
    \[
    \Delta = \disc(\Rc_{\hs}(X), \chi, r) \le 2.0005 \cdot \herdisc(\Rc_{\hs}(X)) \le C_{d, r} N^{1/2-1/(2d)}.
    \]
    By \cref{lem:dev-tolerance}, the partition $(X_1,\dots, X_r)$ is a partition with tolerance greater than or equal to
    \[
    \frac{N}{r}-rC_{d, r} N^{1/2-1/(2d)}-1.
    \]

    Therefore, it is sufficient to show that for some $K_{d, r}>0$, if we choose $E = K_{d, r} t^{1/2-1/{2d}}$ and $N = \lceil rt + E\rceil$, that $E/r \ge rC_{d, r} N^{1/2-1/{2d}}+1$.

    Take $\alpha = 1/2 - 1/(2d) < 1/2$.  Since $x^{\alpha}$ is a concave function, we have $N^{\alpha}\le (rt)^{\alpha} + E^{\alpha}+1$.  If we multiply by $rC_{d, r}$, we have
    \[
    rC_{d, r} N^{\alpha} \le rC_{d, r}(rt)^\alpha + rC_{d, r} E^\alpha + rC_{d, r}.
    \]
    Since $t^{\alpha^2} \leq t^\alpha$ for $t \geq 1$ and $\alpha < 1$, we can choose $K_{d, r}$ sufficiently large, depending only on $d$ and $r$, such that
    \[
    rC_{d, r} N^{\alpha} \le rC_{d, r}(rt)^\alpha + rC_{d, r} E^\alpha + rC_{d, r} \le \frac{E}{3r}+\frac{E}{3r}+\left(\frac{E}{3r}-1\right),
    \]
    which finishes the proof.
\end{proof}

\section{Tverberg for hyperplanes with tolerance}

To establish our results for Tverberg with tolerance for hyperplanes, we use a particular parametrization of hyperplanes.  For convenience, we consider the vectors of $\rr^d$ as $d \times 1$ matrices, so that the dot product $\langle x,y\rangle$ can be written as $x^{\top}y$.

Given a hyperplane $H$ in $\rr^d$, we can write it as $H=H(a,b) = \{x \in \rr^d: a^{\top}x=b\}$, where $\|a\| = 1$.  We have $H(a,b) = H(-a,-b)$, and the parametrization is unique up to the sign change.  Let $\Sym_d$ be the space of all $d\times d$ symmetric matrices and let $\Hc_d$ be the space of all hyperplanes in $\rr^d$. Consider the lift

\begin{align*}
    \Phi : \Hc_d & \hooklongrightarrow \Sym_d \times \rr^d \\
    H(a,b)& \longmapsto (aa^{\top}, ba).
\end{align*}

This is well defined since the two ways to parametrize $H$ give us the same image.  The map $\Phi$ is injective.  The matrix $aa^{\top}$ is the matrix corresponding to the orthogonal projection onto $\operatorname{span}(a)$, which uniquely determines $a$ up to sign.  Since $\tr(aa^{\top})=1$ when $a$ is a unit vector, the image of $\Phi$ is contained in an affine space $\mathcal{A}_d \subseteq \Sym_d \times \rr^d$ of dimension $\frac{d(d+3)}{2}-1$.

This lift translates Tverberg-type properties of families of points to the corresponding properties for families of hyperplanes.  Namely, we have the following lemma and corollary.

\begin{lemma}\label{lem:technical-Mq}
    Let $d$ be a positive integer and $X$ be a finite family of hyperplanes in $\rr^d$.  If $(M,c) \in \conv (\Phi(X))$, then there exists $q \in \rr^d$ such that $Mq=c$.  Moreover, every $q$ that satisfies this property also satisfies $q \in D(X)$.
\end{lemma}

\begin{proof}
    Let $X= \{H_1,\dots,H_n\}$.  For each $i\in [n]$, let $H_i = H(a_i,b_i)$ with $\|a_i\|=1$.  Since $(M,c) \in \conv (\Phi (X))$, we can find coefficients $\lambda_i$ for $i\in [n]$ of a convex combination such that $M = \sum_{i=1}^n \lambda_i a_ia_i^{\top}$ and $c = \sum_{i=1}^n\lambda_i b_i a_i$.

    Let us show that $c \in (\ker M)^{\perp}= \im M$.  If $y \in \ker M$, then
    \[
    0 = y^{\top}M y = \sum_{i=1}^n \lambda_i y^{\top} a_ia_i^{\top}y = \sum_{i=1}^n \lambda_i \langle a_i, y\rangle^2.
    \]
    Since each term is non-negative, for each $i \in [n]$ we must have either $\lambda_i = 0$ or $\langle a_i, y \rangle = 0$.  This implies that
    \[
    c^{\top}y = \sum_{i=1}^n \lambda_i b_i (a_i^\top y) = 0,
    \]
as we wanted to show.

For the second statement, let $q \in \rr^d$ be a vector such that $Mq = c$, and let $v \in S^{d-1}$ be a unit vector.  We want to show that the ray starting at $q$ with direction $v$ intersects or is parallel to at least one $X_i$. If $a_i^{\top}v=0$, then the ray is parallel to $X_i$ and we are done.  Suppose for a contradiction that the ray does not intersect any hyperplane $H_i$.  For each $i \in [n]$ the unique $t_i \in \rr$ such that $q + t_i v$ intersects $H_i$ is
\[
t_i = \frac{b_i - a_i^{\top}q}{a_i^{\top}v}.
\]
The condition above implies $t_i < 0$ for all $i$, which in turn implies that $(v^{\top}a_i)(b_i - a_i^{\top}q)=(b_i - a_i^{\top}q)(a_i^{\top}v)<0$ for all $i \in [n]$.  Since $\sum_{i=1}^n \lambda_i=1$, at least one $\lambda_i$ is positive.  This implies that
\begin{align*}
    0 > \sum_{i=1}^n \lambda_i (v^{\top}a_i)(b_i - a_i^{\top}q) = v^{\top} \left( \sum_{i=1}^n \lambda_ib_i a_i - \sum_{i=1}^n \lambda_i a_i a_i^{\top}q\right) = v^{\top}(c - Mq) = 0,
\end{align*}
which is the contradiction we wanted to find.
\end{proof}

This leads us to the following corollary.

\begin{corollary}\label{cor:family-conv-prd}
    Let $r,d$ be positive integers and let $X_1,\dots,X_r$ be finite families of hyperplanes in $\rr^d$.  If we have
    \(
    \bigcap_{j=1}^r \conv (\Phi (X_j)) \neq \emptyset\) we must also have \(\bigcap_{j=1}^r D(X_j)\neq \emptyset\). 
\end{corollary}

\begin{proof}
Let $(M,c) \in \bigcap_{j=1}^r \conv (\Phi (X_j))$.  Let $q \in \rr^d$ be such that $Mq=c$, which exists by \cref{lem:technical-Mq}.  We also know by \cref{lem:technical-Mq} that $q \in D(X_j)$ for all $j \in [r]$, which implies $\bigcap_{j=1}^r D(X_j) \neq \emptyset$.
\end{proof}

\cref{cor:family-conv-prd} essentially states that any Tverberg-type result for families of points implies a Tverberg-type result for families of hyperplanes, at the cost of increasing the dimension to $(d(d+3)/2)-1$.  This is because, given a family of hyperplanes in $\rr^d$, we can lift them as points to $\mathcal{A}_d$ using $\Phi$, apply a Tverberg-type result in $\mathcal{A}_d$, and then reinterpret the result in $\rr^d$.  For example, we immediately obtain the following result.

\begin{corollary}
    Let $d,r$ be positive integers, then $N_\text{hyp}(t, d, r) = rt + \tilde{O}_{d,r}(\sqrt{t})$.
    
\end{corollary}

\begin{proof}
    We apply the bounds from \cite{Soberon2018}.
\end{proof}

The theorem above holds since the $O(\cdot)$ hides the (polynomial) effect of $d$ and $r$.  We can of course improve this using \cref{thm:main-points}, but we aim for more.  The direct application of \cref{thm:main-points} would give us a bound of the form $\sim rt + O(t^{1/2-1/[d(d+3)-2]})$.  We can improve this by carefully analyzing the lift $\Phi$.

Given a finite family $X$ of hyperplanes in $\rr^d$, we consider $(X,\Rc_{\Phi}(X))$ the range space where we take the ranges induced by closed halfspaces in $\mathcal{A}_d$ applied to $\Phi(X)$.

\begin{lemma}\label{lem:dual-shatter}
    There is an absolute constant $A_0$ such that for every $d \ge 2$, every finite family of hyperplanes $X$ in $\rr^d$ and every $m \ge 1$, we have
    \[
    \pi^*_{\Rc_{\Phi}(X)}(m) \le (A_0 m)^d.
    \]
\end{lemma}

\begin{proof}
    Suppose we have $m$ ranges in $\Rc_{\Phi}(X)$.  We are going to cover $\Hc_d$ with $d$ closed sets, and show that in each of those sets, the $m$ ranges define at most $(Cm)^d$ regions, which will be enough to prove the lemma.

    For $h \in [d]$ the $h$-th region $M_h \subset \Hc_d$ is the set of hyperplanes $H(a,b)$ such that $a_h = \max_{i \in [d]}|a_i|$, where $a=(a_1,\dots,a_d)$.  Notice that every hyperplane is contained in at least one $M_h$, since at least one of its parametrizations $H(a,b)$ or $H(-a,-b)$ must have one of the largest coordinates in absolute value be positive.

    For each $H=H(a,b) \in M_h$, let $z_i = \frac{a_i}{a_h}$ for $i \neq h$, $\beta = \frac{b}{a_h}$, and $w = \frac{1}{a_h}a$.  This operation is easily reversible, and given $w$ we can find $a = w / \|w\|$ and $b = \beta / \|w\|$.

    A halfspace in $\mathcal{A}_d \subset \Sym_d \times \rr^d$ corresponds to the pairs $(a,b)$ such that some linear function on the coefficients of $(aa^{\top},ba)$ evaluates to a non-negative value.  If we multiply this function by $\|w\|^2$, we obtain a polynomial of degree $2$ on $d$ variables: $\beta$ and all $z_i$ for $i\neq h$.

    For $m$ real polynomials of degree $2$ in $d$ variables, the number of sign patterns in $\{-1,0,1\}^m$ realized in $\rr^d$ is at most $(Cm)^d$ for some absolute constant $C$ \cite{Matousek2010}*{Thm 5.5}, \cite{Pollack1993}.  The collection of all such sign patterns over all $M_h$ is an upper bound for the value of $\pi^*_{\Rc_{\Phi}(X)}(m)$.  Therefore, this is bounded above by $(A_0 m)^d$ for some absolute constant $A_0$, as we wanted to show.
\end{proof}

The lemma above is the key point in our analysis of hyperplanes.  The space $\mathcal{A}_d$ has dimension $D=(d(d+3)/2) - 1$, which makes it immediate that $\pi^*_{\Rc_{\Phi}(X)}(m) = O(m^{D})$.  The precise structure of the image of $\Phi$ allows us to save on the exponent.

We are now ready to bound the hereditary discrepancy of $\Rc_{\Phi}$.

\begin{corollary}\label{cor:hyp-her-bound}
    For every $d \geq 2$ and every finite family $X$ of hyperplanes in $\rr^d$, we have
    $$\herdisc(\Rc_\Phi(X)) \leq C_d|X|^{1/2 - 1/(2d)}\sqrt{\log(|X|+1)},$$
    where $C_d$ is a constant depending only on $d$.
\end{corollary}

\begin{proof}
    Apply \cref{thm:mww-expl} with $s = d$ and the bound on the dual shatter function from \cref{lem:dual-shatter}. All estimates hold uniformly under restriction.
\end{proof}

Now we are ready to prove \cref{thm:main-hyperplanes}.

\begin{theorem}
    For positive integers $t, d, r$, we have
    $$N_\text{hyp}(t, d, r) \leq rt + C_{d, r}t^{1/2 - 1/(2d)}\sqrt{\log(t + 1)},$$
    where $C_{d, r}$ is a constant depending only on $d$ and $r$.
\end{theorem}

\begin{proof}
    Let $X$ be a family of hyperplanes in $\rr^d$ and let $X' = \Phi(X)$.  By \cref{cor:hyp-her-bound}, we know that 
    \[
    H = \herdisc(\Rc_\Phi(X)) \leq C_d|X|^{1/2 - 1/(2d)}\sqrt{\log(|X|+1)}.
    \]
    Choose a partition $X_1' \sqcup \cdots \sqcup X_r'$ of $X'$ by \cref{lem:two-r-color}, with corresponding coloring $\chi': X' \to [r]$, and let us denote again
    \[
    \Delta = \disc(\Rc_{\hs}(X'), \chi', r).
    \]
    We know by \cref{lem:two-r-color} (using $3$ instead of $2.0005$ for simplicity) that $\Delta \leq 3 \cdot H$.  The discrepancy of $\Rc_\Phi$ is the discrepancy of the range space induced by halfspaces of $\mathcal{A}_d$ on $X' = \Phi(X)$.  Therefore, we can apply \cref{lem:dev-tolerance} in $\mathcal{A}_d$ and obtain that $\bigcap_{j=1}^r\conv(X'_j \setminus C) \neq \emptyset$ for all $C \subset X'$ of at most $t' = \left\lceil\frac{|X|}{r}-r\Delta \right\rceil-1$.  If $X_j \subseteq X$ is the family such that $\Phi(X_j) = X'_j$, by \cref{cor:family-conv-prd} we have that $\bigcap_{j=1}^r D(X_j \setminus C) \neq \emptyset$ for all sets $C \subseteq X$ of at most $t'$ hyperplanes.

    Finally, there exists a constant $C_{d, r}$ depending only on $d$ and $r$ such that as long as $|X| \ge rt + C_{d, r}t^{1/2 - 1/(2d)}\sqrt{\log(t + 1)}$, we have $t' \ge t$.
    
\end{proof}

\section{Lower bounds for tolerance}

The goal of this section is to prove the lower bound for \cref{thm:main-points}.  We use the discrepancy results from Chazelle, Matou\v{s}ek, and Sharir \cite{Chazelle1995} together with the following simple observation.

\begin{lemma}\label{lem:first-tolerance-bound}
    Let $r,t,p$ be positive integers and $X \subset \rr^d$ be a set of $rt+p$ points.  If a partition $X_1\sqcup \dots \sqcup X_r$ with corresponding coloring $\chi: X \to [r]$ has spread $\spr(\Rc_{\hs}(X), \chi, r) \ge p$, then it is not a Tverberg partition with tolerance $t$.
\end{lemma}

\begin{proof}
    Let $H^+$ be the closed halfspace that realizes the maximum spread.  Assume without loss of generality that $|X_r \cap H^+|\ge p + |X_1 \cap H^+|.$  Consider the numbers $|X_1|,|X_2|,\dots,|X_{r-1}|,(|X_r|-p)$.  Their sum is $rt$, so at least one term must be at most $t$.  If this is $|X_j|$ for $j \neq r$, we can remove all the points of $X_j$.  This would make the $j$-th part empty, which gives us a partition that is not Tverberg.  Therefore $|X_r| - p \le t$.  In this case, let us remove all points of $X_1$ in $H^+$ and all points of $X_r$ in the complement of $H^+$.  We have removed at most $|X_1 \cap H^+| + |X_r \cap (H^+)^c| \le (|X_r\cap H^+|-p)+ |X_r \cap (H^+)^c| \le |X_r| - p \le t$ points.  By doing so, the boundary of $H^+$ now separates what is left in $X_1$ and $X_r$, showing that the partition is not a Tverberg partition.
\end{proof}

Now we construct a family of $n$ points in $\rr^d$ with large spread with respect to $\Rc_{\hs}$.  The set of points will essentially be a grid of $\sim n^{1/d}$ points in each direction.  Formally, let $M = \lceil n^{1/d}\rceil $ and consider the set
\[
K=\left\{\frac{1}{2dM}\left( {z_1}, \dots, {z_d}\right) \in \R^d \mid z_i \in \Z,\ 0 \leq z_i < M\right\} \subseteq \left[0,\frac{1}{d}\right]^d
\]
We take $X_n$ to be a set of $n$ points in $K$.  This set lies in the box $[0,1/d]^d$.  If we denote by $\rho$ the minimum distance between two points in $X_n$, we have \[
\rho = \min_{\substack{x \neq y \\ x, y \in X_n}} \|x-y\|\ge \frac{1}{2dM}\ge \frac{1}{4d}n^{-1/d}.
\]

The main result to establish the spread will be the following result by Chazelle et al.

\begin{theorem}[Chazelle, Matou\v{s}ek, Sharir 1995, \cite{Chazelle1995}*{Section 3}]\label{thm:useful-chazelle}
    Let $m,d$ be positive integers and $A \subset \rr^d$ be a subset of $m$ points of $\rr^d$.  Let $\operatorname{diam}(A)$ be the largest possible distance between two points of $A$, and $\operatorname{radius}(A)$ be the shortest possible distance between any two points of $A$.  Suppose that $\operatorname{diam}(A)/\operatorname{radius}(A) \le Cm^{1/d}$ for some positive constant $C$.  Then, for any two-coloring of $A$ there exists a half-space $H$ within which one color outnumbers the other by at least $c m^{1/2-1/(2d)}$, for some constant $c>0$ that depend on $C$ and $d$.
\end{theorem}

Now we are able to prove the lower bound for \cref{thm:main-points}.

\begin{theorem}
    Let $t,d,r$ be positive integers with $d\ge 2,\ r \geq 2$.  Then, we have
    \[
    N_{\pt}(t,d,r) \ge rt + c_{d, r}t^{1/2 -1/(2d)},
    \]
    where $c_{d, r} > 0$ depends only on $d$ and $r$.
\end{theorem}

\begin{proof}
    We can assume that $d,r \ge 2$, since the cases $d=1$ and $r=1$ are either solved or trivial, respectively.

    Let $n$ be a positive integer and let $X_n$ be the set as defined above.  Let $X_1 \sqcup \dots \sqcup X_r = X_n$ be any partition of $X_n$ with correspondng coloring $\chi: X \to [r]$.  Without loss of generality, we may assume that $X_1$ and $X_2$ are the largest parts of the partition.  Therefore, for $A = X_1 \cup X_2$ we have $m:=|A| \ge 2n/r$.  Moreover, by construction $\operatorname{diam}(A)/\operatorname{radius(A)} \le Cm^{1/d}$, which means we can apply \cref{thm:useful-chazelle} and obtain

    \begin{align*}
        \spr(\Rc_{\hs}(X), \chi, r) \ge \spr(\Rc_{\hs}(A), \chi|_A, 2)&\ge c_{d}(2n/r)^{1/2-1/(2d)} \\ &\ge b_{d,r}n^{1/2-1/(2d)}
    \end{align*}

    for some constant $b_{r,d}>0$.  Now let $p = \lfloor b_{d,r}(rt)^{1/2-1/(2d)} \rfloor$ and set $n = rt+p$.

    Then, $b_{r,d}n^{1/2-1/(2d)} \ge b_{r,d}(rt)^{1/2-1/(2d)} \ge p$.
    In particular, we can now apply \cref{lem:first-tolerance-bound} and finish the proof.
\end{proof}

\bibliography{refref}

\end{document}